\documentclass[final,leqno,letterpaper]{etna}

\setbibdata{1}{xx}{46}{2017} 

\usepackage{graphicx} 
\usepackage{amsmath,amstext,amssymb,bm}
\usepackage{color}

\newtheorem{remark}{Remark}[section]

\newtheorem{assumption}{Assumption}[section]
\allowdisplaybreaks[4]

\newcommand{\p}{{\partial}}  

\newcommand{\nab}{\nabla}

\newcommand{\mct}{\mathcal{T}_h}

\newcommand{\bH}{{\bm H}}
\newcommand{\bL}{{\bm L}}
\newcommand{\Div}{{\rm div}\,}

\newcommand{\bv}{{\bm v}}

\newcommand{\bV}{{\bm V}}
\newcommand{\bX}{{\bm X}}

\newcommand{\bS}{{\bm S}}
\newcommand{\bI}{{\bm I}}
\newcommand{\bz}{\bm z}
\newcommand{\be}{{\bm e}}
\newcommand{\bbP}{\mathbb{P}}

\newcommand{\bpsi}{\bm \psi}
\newcommand{\bY}{\bm Y}

\newcommand{\bbR}{\mathbb{R}}

\newcommand{\bw}{\bm w}
\newcommand{\bn}{\bm n}

\DeclareMathOperator*{\argmin}{arg\,min}

\title{Pressure-robustness in quasi-optimal a priori estimates for the Stokes problem}
\author{Alexander Linke\thanks{Weierstrass Institute for Applied Analysis and Stochastics, Berlin, Germany (alexander.linke@wias-berlin.de)}
\and{Christian Merdon}\thanks{Weierstrass Institute for Applied Analysis and Stochastics, Berlin, Germany (Christian.Merdon@wias-berlin.de)}
\and{Michael Neilan}\thanks{Department of Mathematics, University of Pittsburgh, Pittsburgh, PA 15260 (neilan@pitt.edu). Neilan was partial supported
by the NSF through grant DMS--1719829.}}

\shorttitle{Pressure-robust quasi-optimality for the Stokes problem}
\shortauthor{A.~Linke, C.~Merdon, and M.~Neilan}

\begin{document}

\maketitle

\begin{abstract}
Recent analysis of the divergence
constraint in the incompressible Stokes/Navier--Stokes problem
has stressed the importance
of  equivalence classes of forces and how it plays a fundamental role 
for an accurate space discretization.  Two forces in the momentum balance are
velocity--equivalent if they lead to the same velocity solution,
i.e., if and only if the forces differ by only a 
gradient field.
Pressure-robust space discretizations are designed to
respect these equivalence classes. 
One way to achieve pressure--robust schemes
is to introduce a non--standard discretization of the right--side
forcing term for any inf--sup stable mixed finite element method.
This modification leads to pressure--robust and optimal--order
discretizations, but 
a proof was only available for smooth situations and remained open in the case of minimal regularity, where it cannot be
assumed that the vector Laplacian of the velocity is at least
square-integrable. This contribution closes this gap by
delivering a general estimate for the consistency error that
depends only on the regularity of the data term.
Pressure-robustness of the estimate is achieved by the fact that
the new estimate only depends on the $L^2$ norm of the
Helmholtz--Hodge projector of the data term
and not on the $L^2$ norm of the entire data term. Numerical examples illustrate the theory.

\end{abstract}

\section{Introduction}
Classical mixed finite element theory for the steady
Stokes problem
\begin{equation} \label{eq:cont:stokes}
\begin{split}
  -\nu \Delta \bv + \nabla p & = \bm f, \\
   -\Div \bv & = g
\end{split}
\end{equation}
with inhomogeneous Dirichlet boundary data,
$\bm f \in \bm L^2(\Omega)$ and $g \in L^2_0(\Omega)$ emphasizes 
that the divergence constraint
$-\Div\bv=g$
requires an appropriate discrete mimicking of the
{surjectivity} of the divergence operator
$\mathrm{div} \!\!: \bH^1_0(\Omega) \to L^2_0(\Omega)$
in order
to guarantee optimal convergence properties, see e.g.\ \cite{MR3097958,jlmnr:sirev}.
Recently it has been stressed that
the divergence constraint in the Stokes
problem
 naturally induces a semi-norm and
corresponding equivalence classes of forces, which
require a \emph{second challenge} for an accurate
space discretization:
two forces
$\bm f_1 \in \bm L^2(\Omega)$ and
$\bm f_2 \in \bm L^2(\Omega)$ are \emph{velocity-equivalent}
\cite{2018arXiv180810711G}
\begin{equation}
  \bm f_1 \simeq \bm f_2,
\end{equation} 
if they lead to the same velocity solution $\bv$ in
the Stokes problem \eqref{eq:cont:stokes} --- and this 
happens if and only if both forces differ by a 
gradient field \cite{jlmnr:sirev,MR3824769}, i.e., 
\begin{equation}
  \bm f_1 \simeq \bm f_2  \qquad \Leftrightarrow
   \qquad \exists{\phi \in H^1(\Omega)/\mathbb{R}} : \bm f_2 = \bm f_1 + \nabla \phi.
\end{equation}
The argument is straightforward: denote by
$(\bv_1, p_1)$ and
$(\bv_2, p_2)$ the pairs of velocity and pressure
solutions corresponding to the forces $\bm f_1$ and
$\bm f_2=\bm f_1 + \nabla \phi$.
Then, the difference of the solutions
$(\delta \bv, \delta p) := (\bv_2 - \bv_1, p_2-p_1)
\in \bm H^1_0(\Omega) \times L^2_0(\Omega)$
fulfills the incompressible Stokes equations
$-\nu \Delta (\delta \bv) + \nabla (\delta p) = \nabla \phi$,
$\mathrm{div} (\delta \bv)=0$ with
homogeneous Dirichlet boundary data.
This problem has
the unique solution
$(\delta \bv, \delta p) = (\bm 0, \phi)$, and
thus 
$\bm f_1$ and $\bm f_2 = \bm f_1 + \nabla \phi$
are \emph{velocity-equivalent}
due to $\delta \bv = \bm 0$.

In conclusion one observes that the
velocity solution $\bv$ of \eqref{eq:cont:stokes}
is determined by the following data:
\begin{enumerate}
\item  Dirichlet boundary data,
\item the data $g$,
\item and the Helmholtz--Hodge projector of the data $\bm f$,
which is defined by
$$
  \mathbb{P}(\bm f)
    := \argmin_{\phi \in H^1(\Omega)} \| \bm f - \nabla \phi
      \|_{\bm L^2(\Omega)}, 
$$
\end{enumerate}
while the data term $\bm f - \mathbb{P}(\bm f)$
only influences the pressure.

The recently introduced notion \emph{pressure-robustness}
\cite{LM2016}
allows to discriminate between space discretizations
for \eqref{eq:cont:stokes}, whose discrete velocity
solutions $\bv_h$ depend  on $\mathbb{P}(\bm f)$
and not on the entire data $\bm f$.  Such schemes lead to
a priori error estimates for the discrete velocity
that depend only on $\bv$ and not on $(\bv, \frac{1}{\nu} p)$
--- as in nearly all classical mixed finite element methods
\cite{jlmnr:sirev}.

This contribution focuses now on applying the improved understanding of relevant data in the Stokes problem,
in order to derive a priori error estimates
for various discretely inf--sup stable mixed methods in
cases of \emph{minimal regularity}. A special focus is set
on a recent modified pressure-robust mixed method
\cite{LMT15, LLMS2017}, where the modification introduces a consistency
error that can be optimally estimated in a straightforward manner by $C h^k | \Delta \bv |_{H^{k-1}(\Omega)}$
provided that $\bv \in \bH^{k-1}(\Omega)$. For the lowest--order methods $(k=1)$ this requires $\Delta \bv \in \bL^2(\Omega)$.
In situations of \emph{minimal regularity}, i.e.,
$\bv \in \bH^{1+s}(\Omega)$ with $0 < s < 1$,
we provide an estimation of the consistency error
by a more sophisticated argument involving the
Helmholtz--Hodge projector of the data
$\nu^{-1}\mathbb{P}(\bm f)$. This term is obviously in $\bL^2(\Omega)$, whenever it holds $\bm f \in \bL^2(\Omega)$
and it is shown to be equal to $\bbP(- \Delta \bv)$.
Thus,
although it holds in general that
$\Delta \bv \not\in \bL^2(\Omega)$ one can exploit in the numerical analysis that at least the divergence--free part of
$\Delta \bv$ is in $\bL^2(\Omega)$.
{
This observation also leads to a seemingly new estimate
for classical mixed methods, which can be sharper than classical a priori estimates, see Theorem \ref{thm:ConsistencyErrorEstimate_classical}.}
Eventually, all classical conforming
finite element methods yield
an estimate of the form
\begin{align*}
\|\nab (\bv_h-\bS_h(\bv))\|_{L^2(\Omega)} \le C_A  h^s \| \mathbb{P} (-\Delta \bv) \|_{L^2(\Omega)}
+ \frac{C_B h}{\nu} \|\bm f - \bbP (\bm f) \|_{L^2(\Omega)},
\end{align*}
while their pressure-robust siblings allow for estimates of the form
\begin{align*}
\|\nab (\bv_h-\bS_h(\bv))\|_{L^2(\Omega)} \le (C_1 h + C_A h^s) \|\bbP (-\Delta \bv)\|_{L^2(\Omega)},
\end{align*}
with $C_1 > 0$, $C_A > 0$ and $C_B > 0$ are constants that do not
depend on $h$. Note that for divergence-free conforming
methods, see e.g.\ \cite{scott:vogelius:conforming,MR3120580},
it holds $C_1=C_A=C_B=0$, but for them
the only nontrivial part of
the numerical analysis is the proof of the discrete inf-sup
stability.
Further, structurally identical results are obtained
for the classical and a modified pressure-robust
nonconforming Crouzeix--Raviart finite element method.

The rest of this paper is structured as follows. Section~\ref{sec:preliminaries} introduces the Stokes problem
as well as the framework for the modified finite element method
and the assumptions that are crucial for the theoretical results.
Section~\ref{sec:HLP} focusses on the Helmholtz--Hodge projector
and its application in stability estimates.
Section~\ref{sec:stoksprojectors} introduces the continuous and discrete Stokes projectors and their properties.
Section~\ref{sec:estimates_classical} applies the tools of the previous sections to obtain quasi-optimal estimates for classical finite element
methods that only depend on the data. Section \ref{sec:estimates_probust}
does the same for the modified pressure-robust finite element methods
where now the error is additionally independent of the pressure and the inverse of the viscosity $\nu$. Section \ref{sec:estimates_nonconforming} revisits
quasi-optimal and pressure-robust error estimates for the
nonconforming Crouzeix--Raviart finite element method.
Finally we perform some numerical experiments in Section \ref{sec:Numerics}
and compare these empirical results with the theory.

\section{Preliminary results}\label{sec:preliminaries}
This section introduces some notation,
recalls some preliminaries and formulates an assumption that
is fundamental for the presented theory.
We adopt standard space notation and denote
vector--valued functions and vector--valued function spaces
in boldface.  We use 
$(\cdot,\cdot)$ to denote the $L^2$-inner product over $\Omega\subset \bbR^n$,
and by $\langle\cdot,\cdot\rangle$ the duality pairing between some Hilbert space and its dual. 
We denote by $L^2_0(\Omega)$ the Hilbert-space of square-integrable scalar functions with zero average, and
\begin{align*}
\bH({\rm div};\Omega) & = \{\bw\in \bL^2(\Omega):\ \Div \bw\in L^2(\Omega)\},\\
\bH_0({\rm div};\Omega) & = \{\bw\in \bH({\rm div};\Omega):\ \bv\cdot \bn|_{\p \Omega} = 0\},
\end{align*}
where $\bn$ denotes the outward unit normal of $\p\Omega$.

\subsection{Stokes problem and weak elliptic regularity assumption}
In the following, we study finite element methods 
for the model problem:
for $\bm f \in \bm L^2(\Omega)$
seek $(\bv, p) \in \bH^1_0(\Omega) \times L^2_0(\Omega)$ such that it holds
\begin{align}\label{eqn:Stokes_problem}
  - \nu \Delta \bv + \nabla p = \bm f, \quad \text{and} \quad 
  \Div \bv = 0\qquad \text{in }\Omega.
\end{align}
The extension to the more general divergence
constraint $\Div \bv = g$ with $g\in L^2_0(\Omega)$
is straightforward, and we refer to \cite{jlmnr:sirev}
for details.


A weak formulation of the problem is given by:
search for $(\bv, p) \in \bH^1_0(\Omega) \times L^2_0(\Omega)$
such that it holds
\begin{equation} \label{eq:weak:stokes:problem}
\begin{split}
  \nu (\nabla \bv, \nabla \bw) - (p, \Div \bw) & =
    (\bm f, \bw), \\
    (\Div\bv, q) & = 0 
\end{split}
\end{equation}
for all $(\bw, q) \in \bH^1_0(\Omega) \times L^2_0(\Omega)$.

The space of divergence-free $\bH^1_0(\Omega)$ vector fields is denoted as
\begin{equation}
\bV^0 := \{\bw\in \bH^1_0(\Omega):\ \nab \cdot \bw = 0\}.
\end{equation}

\begin{assumption}\label{asmptn:regularity}
Throughout the paper, we assume that the Stokes problem inherits $\bH^{1+s}(\Omega)\times H^s(\Omega)$ elliptic regularity for
some $s \in (0,1]$ and that 
$\nu \| \bv \|_{H^{1+s}(\Omega)} + \|p\|_{H^s(\Omega)}\le C_{\mathrm{ell},s} \| \bm f \|_{L^2(\Omega)}$. 
\end{assumption}

\section{Helmholtz--Hodge projector}\label{sec:HLP}
According to the $\bm L^2$-orthogonal Helmholtz--Hodge decomposition
(see e.g.\ \cite{GR86})
any vector field $\bm f \in \bm L^2(\Omega)$ can be uniquely
decomposed
into 
\begin{align} \label{eqn:HHLdecomposition}
  \bm f = 
  \nabla \alpha + \mathbb{P}(\bm f),
\end{align}
where $\alpha \in H^1(\Omega)/\mathbb{R}$, and
\[
\mathbb{P}(\bm f) 
\in \bm{L}^2_\sigma(\Omega) := \lbrace \bw \in \bL^2(\Omega) : 
(\nabla q, \bw) = 0 \text{ for all } q \in H^1(\Omega) \rbrace
\]
 is the Helmholtz--Hodge projector
of $\bm f$. Note, that the Helmholtz--Hodge projector of ${\bm f}$ is divergence--free and is the
orthogonal $\bL^2$ projection
of $\bm f$ onto $\bm{L}^2_\sigma(\Omega)$,
i.e.,
\begin{align*}
  (\mathbb{P}(\bm f), \bw) = (\bm f,\bw) \qquad \text{for all } \bw \in \bm{L}^2_\sigma(\Omega).
\end{align*}
Moreover, for the Stokes velocity solution $\bv$
it holds
\begin{equation}\label{eqn:weak_solution}
  \nu (\nabla \bv,\nabla \bw) = (\bm f, \bw)  = (\mathbb{P}(\bm f),\bw) \qquad \text{for all } \bw \in \bV^0.
\end{equation}
The
domain of the Helmholtz--Hodge projector
can be extended to $\bH^{-1}(\Omega)$
with range in $(\bV^0)^*$, the space
of bounded linear functionals on $\bV^0$.
Indeed, for every functional
$\bm f \in \bH^{-1}(\Omega)$ the Helmholtz--Hodge projector
can be defined as the \emph{restriction} to $\bV^0$, i.e.,
it holds
\begin{equation} \label{eq:HH:HminusOne}
  <\! \mathbb{P}(\bm f), \bw \!> = \, <\! \bm f, \bw \!>
   \qquad \text{for all } \bw \in \bV^0.
\end{equation}
%
Condition  \eqref{eq:HH:HminusOne} defines \emph{an extension} of
the Helmholtz--Hodge projector from $\bm L^2(\Omega)$
to $\bH^{-1}(\Omega)$. Assume that the functional
$\hat{\bm f} \in \bH^{-1}(\Omega)$ has a 
representation
$\bm f\in \bL^2(\Omega)$ with
$\bm f = \nabla \alpha + \mathbb{P}\bm f$.
Then it holds
for all $\bw \in \bV^0$
\begin{equation*}
  <\! \mathbb{P}(\hat{\bm f}), \bw \!> = <\! \hat{\bm f}, \bw \!> =
  (\bm f, \bw) = (\mathbb{P}(\bm f), \bw).
\end{equation*}

\begin{lemma} \label{lem:delta:v:reg}
Denote by $-\Delta : \bH^1_0(\Omega) \to \bH^{-1}(\Omega)$ via
\begin{equation} \label{eq:laplace:v}
  <\! -\Delta \bw, \bm \psi \!> := (\nabla \bw, \nabla \bm \psi)
  \qquad \text{for all } \bm \psi \in \bH^1_0(\Omega).
\end{equation}
Then the weak velocity solution $\bv$
of  \eqref{eq:weak:stokes:problem} satisfies
\begin{equation}
  \mathbb{P}(-\Delta \bv) = \frac{1}{\nu} \mathbb{P}(\bm f).
\end{equation}
\end{lemma}
\begin{proof}
This follows directly from a combination of \eqref{eqn:weak_solution}
and \eqref{eq:HH:HminusOne}.
\end{proof}

Thus, although the regularity of the functional $-\Delta \bv$
is not better in general than $-\Delta \bv \in \bH^{-1}(\Omega)$, its
divergence--free part $\mathbb{P}(-\Delta \bv)$ has the
better regularity $\bL^2(\Omega)$.

\begin{remark}
We emphasize that Lemma \ref{lem:delta:v:reg} is of
central importance for the derivation
of \emph{pressure-robust} a priori error estimates
in case of minimal regularity.
We also stress that the quantity $\nu^{-1} \mathbb{P}(\bm f)$,
which appears naturally in the analysis of pressure-robust
methods,
does in fact \textit{not} scale with the inverse of $\nu$,
since it only depends on $\bv$.
\end{remark}

An immediate consequence from Lemma
\ref{lem:delta:v:reg}
is the following result that bounds the norm of the
velocity field
by the norm of the Helmholtz--Hodge projector
of the data $\bm f$. 
\begin{lemma}[Continuous stability estimate]
The exact solution of problem ~\eqref{eqn:Stokes_problem} satisfies
$$
  \|\nab \bv \|_{L^2(\Omega)} \leq \frac{C_{PF}}{\nu} \| \mathbb{P}(\bm f) \|_{L^2(\Omega)} = C_{PF}\| \mathbb{P} (\Delta \bv) \|_{L^2(\Omega)}
$$
where $C_{PF}$ is the constant from the Poincar\'{e}--Friedrichs inequality.
\end{lemma}
\begin{proof}
The result follows directly from testing \eqref{eqn:weak_solution}
with $\bw = \bv$ and using the
Poincar\'{e}--Friedrichs inequality.\hfill
\end{proof}
\begin{remark}
Here, we emphasize that
the right hand side of the stability estimate is
given by a \emph{semi-norm of the data $\bm f$}.
This is a crucial point, which arguably has not been fully exploited
in  classical mixed theory \cite{MR3097958, GR86}.
\end{remark}

\section{Notation and setting for conforming finite element methods}
In the following, we introduce some notation for
the finite element methods used in this contribution.
We denote by $\bX_h\times Q_h$,
a discretely inf-sup stable finite element pair
\cite{MR3097958} for the Stokes problem
with homogeneous Dirichlet boundary conditions
{with respect to a conforming, shape--regular and simplicial triangulation $\mct$ with $h = \max_{T\in \mct} {\rm diam}(T)$.}
The $L^2$ best approximation onto the discrete pressure space
$L^2_0(\Omega)$ is denoted by $\pi_h: L^2_0(\Omega) \to Q_h$, i.e., for all $r \in L^2_0(\Omega)$
it holds
\begin{equation}\label{eqn:L2ProjDef}
  (\pi_h r, q_h) = (r, q_h) \qquad \text{for all $q_h \in Q_h$.}
\end{equation}
{We assume that $Q_h$
has the approximation property
\begin{equation}\label{eqn:L2ProjApprox}
\|r-\pi_h r\|_{L^2(\Omega)} = \inf_{q_h\in Q_h} \|r-q_h\|_{L^2(\Omega)}\le C_{\pi_h, s} h^{s}\|r\|_{H^{s}(\Omega)}
\end{equation}
for all $r\in H^{s}(\Omega)\cap L^2_0(\Omega)$ and $s\in (0,1]$.
}

Let $\mathrm{div}_h:\bX_h \to Q_h$
with $\mathrm{div}_h = \pi_h \Div$ 
denote
the discrete divergence operator.
Due to the assumed discrete inf--sup stability of the pair $\bX_h\times Q_h$,
$\mathrm{div}_h$ is surjective with bounded right--inverse \cite{MR3097958}.
%
%
We define
the space of
discretely divergence--free functions as
\begin{equation*}
\bV^0_h  :=\{\bv_h\in \bX_h : \mathrm{div}_h \bv_h = 0 \}.
\end{equation*}


\subsection{Some modified finite element methods}
As shown in  \cite{Lin14, LMT15, LLMS2017},
a certain modification of the discrete right--hand side of 
the incompressible Stokes problem  renders
 inf-sup stable mixed   methods
{pressure-robust}.
These {pressure-robust} finite element methods
employ a reconstruction operator with the properties stated in the following assumption.

\begin{assumption}\label{asmptn:reconstructionoperator}
We assume that there exists 
an auxiliary finite element space $\bY_h\subset \bH_0({\rm div};\Omega)$
and a  reconstruction operator $\bI_h:\bH^1_0(\Omega)\to \bY_h$ such that
\begin{align}\label{eqn:IhApprox_divfree}
& (i) & 
\Div (\bI_h \bv_h) & = \mathrm{div}_h \bv_h & \text{for all } \bv_h\in \bX_h,\\
& (ii) & \|\bv_h - \bI_h \bv_h\|_{L^2(\Omega)} &\le C_1 h \|\nab \bv_h\|_{L^2(\Omega)} & \text{for all } \bv_h\in \bX_h,\label{eqn:IhApprox}
\end{align}
where $C_1$ depends only on the shape regularity of the mesh.
\end{assumption}

The modified finite element method for the Stokes problem applies 
the reconstruction operator in the right-hand side.
The resulting scheme  seeks $(\bv_h,p_h)\in \bX_h\times Q_h$ such that
\begin{alignat}{2}
\nu (\nab \bv_h,\nab \bw_h) - (\mathrm{div}_h \bw_h,p_h) & = ({\bm f},\bI_h \bw_h)\qquad &&\text{for all } \bw_h\in \bX_h,\label{eqn:discrete_problem}\\
(\mathrm{div}_h \bv_h,q_h) & = 0 \qquad &&\text{for all } q_h\in Q_h.\nonumber
\end{alignat}

Testing \eqref{eqn:discrete_problem} with
discretely divergence-free velocity test functions yields
\begin{equation}
  \label{eqn:weak_solution_discrete}
  \nu  (\nab \bv_h,\nab \bw_h) = (\bm f, \bI_h \bw)  = (\mathbb{P}(\bm f), \bI_h \bw_h) \qquad  \text{for all } \bw_h \in \bV^0_h,
\end{equation}
since for $\bw_h \in \bV^0_h$ it holds $\bI_h \bw_h \in \bm{L}^2_\sigma(\Omega)$.
This last identity is characteristic for pressure-robustness
and in general \emph{not} true for non-divergence-free classical finite element methods. It tells us that the discrete velocity solution
$\bv_h$ of \eqref{eqn:discrete_problem} depends on the appropriate continuous data
$\nu^{-1} \mathbb{P}(\bm f)$ of the problem.

In the case of discontinuous pressure spaces $Q_h$, the
standard interpolation operators of the 
Raviart-Thomas or Brezzi--Douglas--Marini finite element spaces can be employed as a reconstruction operator $\bI_h$,
see \cite{jlmnr:sirev,LM2016, LMT15} for details.
For instance, in the case of the Bernardi--Raugel finite element method \cite{BernardiRaugel}, the standard interpolator into the BDM space of order one can be used. 
For continuous pressure spaces, the design of the reconstruction operator is more involved; see \cite{LLMS2017}
for details in case of the Taylor--Hood or MINI finite element family.

\begin{remark}
Note, that for $\bI_h = \bm 1$ (the identity operator) in \eqref{eqn:discrete_problem} the classical finite element method is obtained. However, 
only divergence-free $H^1$-conforming classical finite element methods, see e.g.\ \cite{scott:vogelius:conforming,MR3120580}, satisfy Assumption~\ref{asmptn:reconstructionoperator}
with $C_1 = 0$. In the results below it will be
specified which results rely on this assumption.
\end{remark}

\begin{lemma}[Discrete stability estimates]
Let $(\bv_h,p_h)\in \bX_h\times Q_h$ satisfy \eqref{eqn:discrete_problem}
and write $\bm f = \nab \alpha+\mathbb{P}\bm f$.
Then if the discrete scheme satisfies Assumption \ref{asmptn:reconstructionoperator}, it holds the estimate
$$
  \|\nab \bv_h \|_{L^2(\Omega)} \leq  (C_{PF} + C_1h) \| \nu^{-1} \mathbb{P}(\bm f) \|_{L^2(\Omega)}
  = (C_{PF} + C_1h) \| \mathbb{P} (\Delta \bv) \|_{L^2(\Omega)}.
$$
If the discrete scheme with $\bI_h = \bm 1$ does not satisfy Assumption \ref{asmptn:reconstructionoperator}, it only holds
\begin{align*}
\|\nab \bv_h \|_{L^2(\Omega)} 
&\leq C_{PF} \| \nu^{-1} \mathbb{P}(\bm f) \|_{L^2(\Omega)}
+ \frac{1}{\nu} \| \alpha - \pi_h \alpha \|_{L^2(\Omega)}\\
&\leq C_{PF} \| \mathbb{P} (\Delta \bv) \|_{L^2(\Omega)}
+ \frac{C h}{\nu} \| \bm f - \mathbb{P}(\bm f) \|_{L^2(\Omega)}.
\end{align*}
\end{lemma}
\begin{proof}
Testing \eqref{eqn:weak_solution_discrete} with $\bw_h = \bv_h$, a discrete Poincar\'e--Friedrichs inequality and \eqref{eqn:IhApprox} yield
\begin{align*}
\nu \|\nab \bv_h \|_{L^2(\Omega)}^2
= (\mathbb{P}(\bm f), \bI_h \bv_h )
& \leq \| \mathbb{P}(\bm f) \|_{L^2(\Omega)}
\left( \| \bv_h \|_{L^2(\Omega)} + \| \bv_h - \bI_h \bv_h\|_{L^2(\Omega)} \right)\\
& \leq (C_{PF} + C_1h) \| \mathbb{P}(\bm f) \|_{L^2(\Omega)} \|\nab \bv_h \|_{L^2(\Omega)}.
\end{align*}
If $\bI_h = \bm 1$ and Assumption \ref{asmptn:reconstructionoperator} is not satisfied then inserting the Helmholtz--Hodge decomposition of $\bm f$
and an integration by parts give
\begin{align*}
\nu \|\nab \bv_h \|_{L^2(\Omega)}^2
& = (\mathbb{P}(\bm f), \bv_h ) + (\nabla \alpha, \bv_h)\\
& = (\mathbb{P}(\bm f), \bv_h )
 -(\alpha, \Div \bv) \\
& = (\mathbb{P}(\bm f), \bv_h )
 -(\alpha - \pi_h \alpha, \Div \bv) \\
& \leq \left(C_{PF} \| \mathbb{P}(\bm f) \|_{L^2(\Omega)}
+ \| \alpha - \pi_h \alpha \|_{L^2(\Omega)}
 \right) \| \nabla \bv_h \|_{L^2(\Omega)}.
\end{align*}
Property \eqref{eqn:L2ProjApprox}
shows $\| \alpha - \pi_h \alpha \|_{L^2(\Omega)} \leq C h \| \nab \alpha \|_{L^2(\Omega)} = C h \| \bm f - \mathbb{P}(\bm f) \|_{L^2(\Omega)}$.
This concludes the proof.
\end{proof}

\section{Continuous and discrete Stokes projectors}\label{sec:stoksprojectors}
In preparation for the a priori error estimates,
this section studies the continuous and the discrete Stokes projectors. They are defined as the $\bH^1$-seminorm 
best-approximations into the (discretely) divergence-free functions, i.e.\ $\bS_h:\bH^1_0(\Omega)\to \bV^0_h$
and $\bS:\bX_h\to \bV^0$ are defined by
\begin{alignat}{2}
\label{eqn:DefSh}
(\nab \bS_h(\bv),\nab \bw_h) &= (\nab \bv,\nab \bw_h)\qquad &&\forall \bw_h\in \bV^0_h,\\
\label{eqn:DefS}
(\nab \bS(\bv_h),\nab \bw) & = (\nab \bv_h,\nab \bw)\qquad &&\forall \bw\in \bV^0.
\end{alignat}
The rest of this section collects useful properties of these projectors.

\begin{lemma}[Stokes projector identity]\label{lem:StokesProjectoridentity}
For any $\bv \in \bH^1_0(\Omega)$ and $\bv_h \in \bX_h$, it holds the identity
\[
(\nab \bS_h(\bv),\nab \bw_h) = (\nab \bv,\nab \bS(\bw_h))\qquad \forall \bv\in \bV^0,\ \bw_h\in \bV^0_h.
\]
\end{lemma}
\begin{proof}
  This follows directly from the combination of the definitions of $\bS_h$ and $\bS$.
\end{proof}


\begin{lemma}\label{lem:Duality}
Suppose that the Stokes problem
satisfies Assumption~\ref{asmptn:regularity}.
Then there holds
\begin{align}\label{eqn:DualityEstimate}
\|\bw_h- \bS(\bw_h)\|_{L^2(\Omega)}\le C_2 h^s\|\nab \cdot \bw_h\|_{L^2(\Omega)}\qquad \forall \bw_h\in \bV^0_h.
\end{align}
\end{lemma}
\begin{proof}
Let $(\bpsi,r)\in \bH^1_0(\Omega)\times L^2_0(\Omega)$ solve
the Stokes problem with source $\bw_h - \bS(\bw_h)$ and unit viscosity:
\begin{alignat*}{2}
(\nab \bpsi,\nab \bz) - (\nab \cdot \bz,r) & = (\bw_h -\bS(\bw_h),\bz)\qquad &&\forall \bz\in \bH^1_0(\Omega),\\
(\nab \cdot \bpsi,q) & = 0\qquad &&\forall q\in L_0^2(\Omega).
\end{alignat*}
Testing the first equation with $\bz = \bw_h - \bS(\bw_h)$
and employing \eqref{eqn:DefS} leads to
\begin{align*}
\|\bw_h - \bS(\bw_h)\|_{L^2(\Omega)}^2 
&= (\nab \bpsi,\nab (\bw_h-\bS(\bw_h))) - (\nab \cdot (\bw_h-\bS(\bw_h)),r)\\
&= - (\nab \cdot \bw_h,r).
\end{align*}
%
Recall that $\pi_h r$ is the $L^2$--projection of $r$ defined by \eqref{eqn:L2ProjDef},
and note that
 it holds $(\nab \cdot \bw_h,\pi_h{r}_h)=0$
since $\bw_h\in \bV^0_h$.  Consequently, by \eqref{eqn:L2ProjApprox}, we have
%
\begin{align*}
\|\bw_h - \bS(\bw_h)\|_{L^2(\Omega)}^2 
&= -(\nab \cdot \bw_h,r-\pi_h{r}) \le \|\nab \cdot \bw_h\|_{L^2(\Omega)} \|r-\pi_h{r}\|_{L^2(\Omega)}\\
&\le C_{\pi_h,s} h^s \|\nab \cdot \bw_h\|_{L^2(\Omega)} \|r\|_{H^s(\Omega)}.
\end{align*}
Finally, the elliptic regularity Assumption~\ref{asmptn:regularity} implies
$\|r\|_{H^s(\Omega)}\le C_{\mathrm{ell},s} \|\bw_h-\bS(\bw_h)\|_{L^2(\Omega)}$, and so
\begin{align*}
\|\bw_h - \bS(\bw_h)\|_{L^2(\Omega)}^2 
&\le C_{\mathrm{ell},s} C_{\pi_h, s} h^s \|\nab \cdot \bw_h\|_{L^2(\Omega)} \|\bw_h-\bS(\bw_h)\|_{L^2(\Omega)}.
\end{align*}
Dividing the last inequality by $ \|\bw_h-\bS(\bw_h)\|_{L^2(\Omega)}$ gets the desired result.
\end{proof}


\section{Quasi-optimal a priori error estimates for classical
finite element methods}\label{sec:estimates_classical}
This section derives a priori error estimates
for classical finite element methods that are not pressure-robust, i.e.\
do not satisfy Assumption~\ref{asmptn:reconstructionoperator} with $\bI_h = \bm 1$ like
the Bernardi--Raugel, MINI or Taylor--Hood finite element methods. The proof of the
estimate bounds the error of the best-approximation
by the right-hand side data.

\begin{theorem}\label{thm:ConsistencyErrorEstimate_classical}
Suppose that the Stokes problem satisfies Assumption~\ref{asmptn:regularity}, the reconstruction operator
is taken to be the identity $\bI_h = \bm 1$, and that $\bI_h$ does not satisfy 
Assumption~\ref{asmptn:reconstructionoperator}.
Then there holds
\begin{align*}
\|\nab (\bv_h-\bS_h(\bv))\|_{L^2(\Omega)} &
 \le  C_2  h^s \| \nu^{-1} \mathbb{P} (\bm f) \|_{L^2(\Omega)}
+ \frac{C_{\pi_h,1} h}{\nu}  \|\bm f - \mathbb{P}(\bm f)\|_{L^2(\Omega)} \\
 & = C_2  h^s \| \mathbb{P} (-\Delta \bv) \|_{L^2(\Omega)}
+ \frac{C_{\pi_h,1} h}{\nu}  \|\bm f - \mathbb{P}(\bm f)\|_{L^2(\Omega)}
\end{align*}
with $C_2>0$ given by \eqref{eqn:DualityEstimate}.
\end{theorem}
\begin{proof}
Write $\be_h := \bv_h-\bS_h(\bv)$ and note that $\be_h\in \bV^0_h$. Hence,
it follows from Lemmas~\ref{lem:StokesProjectoridentity} and  \ref{lem:Duality} that
\begin{align}
\|\nab \be_h\|_{L^2(\Omega)}^2  
&= (\nab \bv_h,\nab \be_h) - (\nab \bS_h(\bv),\nab \be_h) \nonumber \\
&= (\nab \bv_h,\nab \be_h) - (\nab \bv,\nab \bS(\be_h)) \nonumber \\
& = \nu^{-1} (\bm f,\be_h-\bS(\be_h)) \nonumber \\
& = \nu^{-1} (\mathbb{P}(\bm f) + \nabla \alpha,\be_h-\bS(\be_h)) \nonumber \\
& = (\mathbb{P} (-\Delta \bv),\be_h-\bS(\be_h))
- \nu^{-1} (\alpha - \pi_h \alpha,\mathrm{div}(\be_h-\bS(\be_h))) \label{eq:rem:ref} \\
& \le
\left( C_2  h^s \| \mathbb{P} (-\Delta \bv) \|_{L^2(\Omega)}
+ \nu ^{-1} \|\alpha - \pi_h \alpha \|_{L^2(\Omega)}
\right) \|\nab \be_h\|_{L^2(\Omega)}, \nonumber
\end{align}
where $\alpha$ stems from the Helmholtz--Hodge decomposition
\eqref{eqn:HHLdecomposition} of $\bm f$.
The best approximation property of $\pi_h \alpha$
shows $\| \alpha - \pi_h \alpha \|_{L^2(\Omega)} \leq C_{\pi_h,1} h \| \nab \alpha \|_{L^2(\Omega)} = C_{\pi_h,1} h \| \bm f - \mathbb{P}(\bm f) \|_{L^2(\Omega)}$.
This concludes the proof.
\end{proof}

\begin{remark}
Classical results for conforming
mixed methods \cite{GR86} show the a priori estimate
\[
\|\nab (\bv_h-\bS_h(\bv))\|_{L^2(\Omega)} \leq
\frac{1}{\nu} \inf_{q_h\in Q_h} \|p-q_h\|_{L^2(\Omega)},
\]
which scales like $\nu^{-1} h^s$ under the given regularity assumptions.
Such an estimate is sometimes sharper than
Theorem \ref{thm:ConsistencyErrorEstimate_classical},
but can also be less sharp. \\
i) If it holds, e.g.,
$p \in Q_h$, then the error on the right hand side
of the classical estimate is zero.
This is also preserved in the computations for the new estimate 
until \eqref{eq:rem:ref}, since in the special case
$\bm f = -\nu \Delta \bv$ \eqref{eq:rem:ref}
can be shown to vanish identically. \\
ii) If it holds $p \not\in Q_h$ and if the solution
$(\bv, p) \in \bH^{1+s}(\Omega) \times H^s(\Omega)$ has a low regularity with $s < 1$, then the new estimate can be sharper
e.g.\ for $\nu \ll 1$,
since it predicts an a priori error $\mathcal{O}(h^s + \nu^{-1} h)$, while the classical estimate predicts an error decay like
$\mathcal{O}(\nu^{-1} h^s)$.
We remark that
the pressure-dependent consistency error is influenced by two different contributions,
one determined by $-\Delta \bv$ and another one
determined by
$\frac{1}{\nu} ({\bm f} - \mathbb{P}({\bm f}))$.
\end{remark}

\begin{theorem}[A priori error estimate]
  Under the assumptions of Theorem~\ref{thm:ConsistencyErrorEstimate_classical}, it holds
  \begin{multline*}  
    \|\nab(\bv-\bv_h)\|^2_{L^2(\Omega)} \le (1+C_F) \inf_{\bw_h\in \bX_h} \|\nab (\bv-\bw_h)\|^2_{L^2(\Omega)}\\
 +  \left(C_2  h^s \| \mathbb{P} (-\Delta \bv) \|_{L^2(\Omega)}
+ \frac{C_{\pi_h,1} h}{\nu}  \|\bm f - \mathbb{P}(\bm f)\|_{L^2(\Omega)}\right)^2.
    \end{multline*}
\end{theorem}
\begin{proof}
  The proof starts with the Pythagoras theorem (using \eqref{eqn:DefSh})
  \begin{align*}
    \|\nab (\bv-\bv_h)\|_{L^2(\Omega)}^2 = \|\nab (\bv-\bS_h(\bv))\|_{L^2(\Omega)}^2
    + \|\nab (\bv_h-\bS_h(\bv))\|_{L^2(\Omega)}^2.
  \end{align*}
  The second term can be estimated by Theorem~\ref{thm:ConsistencyErrorEstimate_classical}
  and the first term can be bounded by the best-approximation error in $\bX_h$
  by the standard argument
\[
\|\nab (\bv-\bS_h(\bv))\|_{L^2(\Omega)}\le \inf_{\bw_h\in \bV^0_h} \|\nab (\bv-\bw_h)\|_{L^2(\Omega)}\le (1+C_F) \inf_{\bw_h\in \bX_h} \|\nab (\bv-\bw_h)\|_{L^2(\Omega)},
\]
where $C_F \geq 1$ denotes the stability constant
of the Fortin operator of the mixed method, see e.g.\
\cite{jlmnr:sirev, GR86}.
\end{proof}

\section{Quasi-optimal pressure-robust a priori error estimates}
\label{sec:estimates_probust}

This section concerns novel quasi-optimal a priori error estimates
for conforming divergence-free and pressure-robustly modified finite element methods. Here, the distance between the discrete solution and the
discrete Stokes projector can be bounded by $\|\bbP (-\Delta \bv)\|_{L^2(\Omega)}$ which is in general much smaller than the bound in Theorem~\ref{thm:ConsistencyErrorEstimate_classical}, especially for small $\nu$.

\begin{theorem}\label{thm:ConsistencyErrorEstimate}
Suppose that the Stokes problem satisfies Assumption~\ref{asmptn:regularity} and that the reconstruction
operator $\bI_h$ satisfies Assumption~\ref{asmptn:reconstructionoperator}.
Then there holds
\begin{align*}
\|\nab (\bv_h-\bS_h(\bv))\|_{L^2(\Omega)} \le (C_1 h +C_2 h^s) \|\bbP(- \Delta \bv)\|_{L^2(\Omega)},
\end{align*}
with $C_1>0$ and $C_2>0$ given by \eqref{eqn:IhApprox} and \eqref{eqn:DualityEstimate}, respectively. Note, that there is
no dependency on $\nu^{-1}$.
\end{theorem}
\begin{proof}
Write $\be_h := \bv_h-\bS_h(\bv)$ and note that $\be_h\in \bV^0_h$. Hence,
\begin{align*}
(\nab \bv_h,\nab \be_h) 
&= \frac1{\nu} ({\bm f},\bI_h \be_h) = (\bbP(-\Delta \bv),\bI_h \be_h)\\
& = (\bbP(-\Delta \bv),\bI_h \be_h-\be_h)+(\bbP(-\Delta \bv),\be_h).
%
\end{align*}

The latter term is split up into (using also Lemma~\ref{lem:StokesProjectoridentity})
\begin{align*}
(\bbP(-\Delta \bv),\be_h) & = (\bbP(-\Delta \bv),\be_h-\bS(\be_h)) + (\bbP(-\Delta \bv),\bS(\be_h))\\
& = (\bbP(-\Delta \bv),\be_h-\bS(\be_h))+ (\nab \bv,\nab \bS(\be_h))\\
& = (\bbP(-\Delta \bv),\be_h - \bS(\be_h)) + (\nab \bS_h(\bv),\nab \be_h).
\end{align*}

It then follows from Lemma  \ref{lem:Duality} 
and \eqref{eqn:IhApprox} that
\begin{align*}
\|\nab \be_h\|_{L^2(\Omega)}^2  
&= (\nab \bv_h,\nab \be_h) - (\nab \bS_h(\bv),\nab \be_h)\\
& = (\bbP(-\Delta \bv),\bI_h \be_h-\be_h)+(\bbP(-\Delta \bv),\be_h-\bS(\be_h))\\
&\le \|\bbP(-\Delta \bv)\|_{L^2(\Omega)} \big( \|\bI_h \be_h - \be_h\|_{L^2(\Omega)}+ \|\be_h-\bS(\be_h)\|_{L^2(\Omega)}\big)\\
&\le (C_1 h +C_2 h^s) \|\bbP(-\Delta \bv)\|_{L^2(\Omega)} \|\nab \be_h\|_{L^2(\Omega)}.
\end{align*}
This concludes the proof.
\end{proof}

\begin{theorem}[A priori error estimate]
  Under the assumptions of Theorem~\ref{thm:ConsistencyErrorEstimate}, it holds
  \begin{align*}  
    \|\nab(\bv-\bv_h)\|^2_{L^2(\Omega)} \le (1+C_F) \inf_{\bw_h\in \bX_h} \|\nab (\bv-\bw_h)\|^2_{L^2(\Omega)} +  \left((C_1 h +C_2 h^s) \|\bbP (-\Delta \bv)\|_{L^2(\Omega)}\right)^2.
    \end{align*}
\end{theorem}
\begin{proof}
  The proof starts with the Pythagoras theorem (using \eqref{eqn:DefSh})
  \begin{align*}
    \|\nab (\bv-\bv_h)\|_{L^2(\Omega)}^2 = \|\nab (\bv-\bS_h(\bv))\|_{L^2(\Omega)}^2
    + \|\nab (\bv_h-\bS_h(\bv))\|_{L^2(\Omega)}^2.
  \end{align*}
  The second term can be estimated by Theorem~\ref{thm:ConsistencyErrorEstimate}
  and the first term can be bounded by the best-approximation error in $\bX_h$
  by the standard argument
\[
\|\nab (\bv-\bS_h(\bv))\|_{L^2(\Omega)}\le \inf_{\bw_h\in \bV^0_h} \|\nab (\bv-\bw_h)\|_{L^2(\Omega)}\le (1+C_F) \inf_{\bw_h\in \bX_h} \|\nab (\bv-\bw_h)\|_{L^2(\Omega)},
\]
where $C_F \geq 1$ denotes the stability constant
of the Fortin operator of the mixed method,
see e.g.\ \cite{jlmnr:sirev, GR86}.
\end{proof}

\section{Estimates for the nonconforming Crouzeix--Raviart finite element method}
\label{sec:estimates_nonconforming}
In this section we 
consider the space $\bX_h \not\subset \bH^1_0(\Omega)$ of nonconforming Crouzeix-Raviart functions, i.e., piecewise affine vector fields that are weakly
 continuous across edges  (2D)
or faces (3D) in the triangulation, see e.g.\ \cite{crouzeix:raviart:1973,crforty:brenner}.
To describe this space in detail we require some notation.
Recall that $\mct$ is a conforming, shape--regular, and simplicial triangulation  
of $\Omega$ parameterized by $h= \max_{T\in \mct} {\rm diam}(T)$.
We denote by $\mathcal{E}_h$ the set of $(n-1)$--dimensional simplices in $\mct$, 
i.e., $\mathcal{E}_h$ is either the set of edges (2D) or faces (3D) in $\mct$.
Let $P_m(T)$ denote the space of polynomials of degree $\le m$ on $T$,
and let ${\bm P}_m(T) = (P_m(T))^n$.
Then the Crouzeix-Raviart  space $\bX_h$ consists of all functions $\bw_h\in \bL^2(\Omega)$
with the properties $\bw_h|_T\in {\bm P}_1(T)$, $\int_E \bw_h$
is single--valued for all $E\in \mathcal{E}_h$, and $\int_E \bw_h =0$
for all boundary $E\in \mathcal{E}_h$.
The discrete pressure space $Q_h$ is the space of piecewise constants
with vanishing mean.  It is well--known that the pair $\bX_h\times Q_h$
is inf--sup stable.

Note that Crouzeix-Raviart functions $\bw_h \in \bV^0_h$ are not divergence-free in a $\bH(\mathrm{div})$-sense (as their normal traces are not continuous), but their piecewise divergence
vanishes.
Possible $\bH(\mathrm{div})$-conforming reconstruction operators $\bI_h$ for this
method are the lowest-order Raviart--Thomas or BDM  interpolation
operators, see \cite{BLMS15} for details.

In order to show the same quasi-optimal a priori error estimates for the Crouzeix--Raviart method some arguments have to be slightly modified.
First, the Stokes projectors $\bS_h:\bH^1_0(\Omega)\to \bV^0_h$
and $\bS:\bX_h\to \bV^0$ are now defined by
using the piecewise gradients $\nabla_h$, i.e.,
\begin{alignat}{2}
\label{eqn:DefShCR}
(\nab_h \bS_h(\bv),\nab_h \bw_h) &= (\nab \bv,\nab_h \bw_h)\qquad &&\forall \bw_h\in \bV^0_h,\\
\label{eqn:DefSCR}
(\nab \bS(\bv_h),\nab \bw) & = (\nab_h \bv_h,\nab \bw)\qquad &&\forall \bw\in \bV^0.
\end{alignat}
Recall the Crouzeix--Raviart Fortin interpolation
\begin{align*}
  \bI_\text{CR} \bv \in \bX_h
  \quad \text{defined by} \quad \int_E \bI_\text{CR} \bv  = \int_E \bv  \quad \text{for all } E \in \mathcal{E}_h,
\end{align*}
which satisfies the approximation property
\begin{equation}\label{eqn:CRApprox}
\|\nab_h (\bv-\bI_\text{CR} \bv)\|_{L^2(\Omega)}\le C_\text{CR} h^s \|\bv\|_{H^{1+s}(\Omega)}
\end{equation}
for all $s \in [0,1]$.
This definition of the interpolant yields the well--known property \cite{crforty:brenner}
\begin{align*}
  \int_T \nabla (\bv - \bI_\text{CR} \bv) = 0
  \quad \text{for all } T \in \mathcal{T}_h
\end{align*}
and in particular $\int_T \mathrm{div} (\bv - \bI_\text{CR} \bv)  = 0$
for any $T \in \mathcal{T}_h$. Since $\nabla_h \bw_h$ is piecewise constant this also reveals that we have $\bS_h = \bI_\text{CR}$, i.e.,
the Crouzeix--Raviart interpolator is the discrete Stokes projector. Also note that the Stokes projector identity holds in the form
\[
(\nab_h \bS_h(\bv),\nab_h \bw_h) = (\nab \bv,\nab \bS(\bw_h))\qquad \forall \bv\in \bV^0,\ \bw_h\in \bV^0_h.
\]
However, in general $\bI_\text{CR} \bv \in \bX_h$ does not imply $\bI_\text{CR} \bv \in \bH(\mathrm{div},\Omega)$ and therefore 
Lemma~\ref{lem:Duality} has to be modified as well.

The analysis also needs another mapping that projects a discretely divergence--free Crouzeix--Raviart function 
to some $H^1$-conforming divergence-free function. Such an operator was introduced in \cite{MR3787384} and is based on rational bubble functions.

\begin{lemma}\label{lem:DualityCR}
Suppose that the Stokes problem satisfies Assumption~\ref{asmptn:regularity}.
Then there holds
\begin{align}\label{eqn:DualityEstimateCR}
\|\bw_h- \bS(\bw_h)\|_{L^2(\Omega)}\le C_3 h^s\|\nab_h \bw_h\|_{L^2(\Omega)}\qquad \forall \bw_h\in \bV^0_h.
\end{align}
\end{lemma}
\begin{proof}
Consider the $H^1_0$-conforming and $H^1$-stable operator $\bm{E}_h$ from \cite{MR3787384} with
the properties
\begin{subequations}
\label{eqn:EhProperties}
\begin{alignat}{2}
  \nab \cdot (\bm{E}_h \bw_h) & = 0 \quad &&\text{for all } \bw_h \in \bV^0_h,\\  (\nab_h \bm{u}_h, \nab (\bm{E}_h \bw_h - \bw_h)) & = 0 \quad &&\text{for all } \bw_h, \bm{u}_h \in \bX_h,\\
 \|\nab \bm{E}_h \bw_h\|_{L^2(\Omega)}+ h^{-1}\|\bm{E}_h \bw_h - \bw_h\|_{L^2(\Omega)}&\le C_{E_h} \|\nab_h \bw_h\|_{L^2(\Omega)}\quad &&\text{for all }\bw_h\in \bX_h.
\end{alignat}
\end{subequations}
The second property follows from \cite[$I_2 = 0$ in proof of Theorem~5.1]{MR3787384}.
As in Lemma~\ref{lem:Duality} we look at the solution
$(\bpsi,r)\in \bH^1_0(\Omega)\times L^2_0(\Omega)$ of
the Stokes problem with modified source $\bm{E}_h \bw_h - \bS(\bw_h)$ and unit viscosity:
\begin{alignat*}{2}
(\nab \bpsi,\nab \bz) - (\nab \cdot \bz,r) & = (\bm{E}_h  \bw_h -\bS(\bw_h),\bz)\qquad &&\forall \bz\in \bH^1_0(\Omega),\\
(\nab \cdot \bpsi,q) & = 0\qquad &&\forall q\in L_0^2(\Omega).
\end{alignat*}
Testing the first equation with $\bz = \bm{E}_h \bw_h - \bS(\bw_h) \in \bH^1_0(\Omega)$ and using \eqref{eqn:DefSCR}, \eqref{eqn:EhProperties}
and \eqref{eqn:CRApprox} leads to
\begin{align*}
\|\bm{E}_h \bw_h - \bS(\bw_h)\|_{L^2(\Omega)}^2 
&= (\nab \bpsi,\nab (\bm{E}_h \bw_h-\bS(\bw_h)))\\
&= (\nab (\bpsi- \bI_{\text{CR}} \bpsi),\nab_h (\bm{E}_h \bw_h - \bw_h))\\
&\leq C_\text{CR} (1+C_{E_h}) h^s \| \bpsi \|_{H^{1+s}(\Omega)} \| \nab_h \bw_h \|_{L^2(\Omega)}.
\end{align*}
The elliptic regularity assumption implies
$\| \bpsi \|_{H^{1+s}(\Omega)} \le C_{\mathrm{ell},s} \|\bm{E}_h \bw_h-\bS(\bw_h)\|_{L^2(\Omega)}$ and yields
\begin{align*}
  \|\bm{E}_h \bw_h - \bS(\bw_h)\|_{L^2(\Omega)}
  \leq C_\text{CR} (1+C_{E_h}) C_{\mathrm{ell},s} h^s \| \nab_h \bw_h \|_{L^2(\Omega)}.
\end{align*}
Finally, a triangle inequality gives
\begin{align*}
  \|\bw_h - \bS(\bw_h)\|_{L^2(\Omega)} & =
  \|\bm{E}_h \bw_h - \bw_h\|_{L^2(\Omega)}
  +\|\bm{E}_h \bw_h - \bS(\bw_h)\|_{L^2(\Omega)}\\
  & \leq (C_{E_h} h + C_\text{CR} (1+C_{E_h}) C_{\mathrm{ell},s} h^s) \| \nab_h \bw_h \|_{L^2(\Omega)}.
\end{align*}
This concludes the proof.
\end{proof}

The previous result and similar arguments as in the conforming case enable us to prove the following theorem.

\begin{theorem}\label{thm:ConsistencyErrorEstimateCR}
Suppose that the Stokes problem satisfies Assumption~\ref{asmptn:regularity} and that the reconstruction
operator $\bI_h$ satisfies Assumption~\ref{asmptn:reconstructionoperator}.
Then there holds
\begin{align*}
\|\nab_h (\bv_h-\bS_h(\bv))\|_{L^2(\Omega)} \le (C_1 h +C_3 h^s) \|\bbP (-\Delta \bv)\|_{L^2(\Omega)},
\end{align*}
with $C_1>0$ and $C_3>0$ given by \eqref{eqn:IhApprox} and \eqref{eqn:DualityEstimateCR}, respectively.
Without Assumption~\ref{asmptn:reconstructionoperator}, a result similar to Theorem~\ref{thm:ConsistencyErrorEstimate_classical} is valid, i.e.,
\begin{align*}
\|\nab_h (\bv_h-\bS_h(\bv))\|_{L^2(\Omega)} \le C_3  h^s \| \mathbb{P} (-\Delta \bv) \|_{L^2(\Omega)}
+ \frac{C_{E_h} h}{\nu} \|\bm f - \bbP (\bm f) \|_{L^2(\Omega)}.
\end{align*}
\end{theorem}
\begin{proof}
The proof of the first result is nearly identical to the proof of Theorem~\ref{thm:ConsistencyErrorEstimate}
with slight changes concerning the application of $\nabla_h$ and the
replacement of Lemma~\ref{lem:Duality} by Lemma~\ref{lem:DualityCR}.
Likewise, the proof of the second result is almost identical to the proof of Theorem~\ref{thm:ConsistencyErrorEstimate_classical}. However, one term has to be estimated differently, as follows.
With $\mathrm{div} \bS(\be_h)=0$ and \eqref{eqn:EhProperties}, it holds
\begin{align*}
  \frac{1}{\nu}(\nabla \alpha, \be_h - \bS(\be_h))
  = \frac{1}{\nu}(\nabla \alpha, \be_h - \bm{E}(\be_h))
  & \leq \frac{1}{\nu}\| \nabla \alpha \|_{L^2(\Omega)} \| \be_h - \bm{E}(\be_h) \|_{L^2(\Omega)}\\
 & \leq \frac{C_{E_h} h}{\nu} \| \bm{f} - \bbP (\bm f) \|_{L^2(\Omega)}.
\end{align*}

\end{proof}

\section{Numerical Example}\label{sec:Numerics}

\begin{table}
\centering
\begin{tabular}{ccccc}
ndof & $ \!\| \nabla(\bv - \bv_h) \|_{L^2(\Omega)} \!$ & order & $ \!\| \nabla(\bv_h - \bS_h \bv) \|_{L^2(\Omega)} \!$ & order\\
\hline
    379 & 1.4151e+00 &  -     & 5.0351e-02 & -     \\
   1414 & 9.7300e-01 &  0.542 & 3.0576e-02 & 0.722 \\
   5458 & 6.7235e-01 &  0.535 & 1.6366e-02 & 0.905 \\
  21442 & 4.6297e-01 &  0.540 & 8.4114e-03 & 0.964 \\
  84994 & 3.1819e-01 &  0.543 & 4.2567e-03 & 0.986 \\
 338434 & 2.1844e-01 &  0.545 & 2.1402e-03 & 0.996 \\
1350658 & 1.4988e-01 &  0.546 & 1.0729e-03 & 1.000 \\
\end{tabular}
\caption{\label{tab:ex1brlr_1c}Errors for the classical Bernardi--Raugel finite element method for $\bm f = \nabla (\sin(xy\pi))$ and $\nu = 1$.}
\end{table}

\begin{table}
\centering
\begin{tabular}{cccc}
ndof & $ \!\| \nabla(\bv - \bv_h) \|_{L^2(\Omega)} \!$ & order & $ \!\| \nabla(\bv_h - \bS_h \bv) \|_{L^2(\Omega)} \!$\\
\hline
    379 & 1.4142e+00 &  -     & 8.5800e-11\\
   1414 & 9.7261e-01 &  0.542 & 1.2467e-13\\
   5458 & 6.7218e-01 &  0.535 & 1.9887e-14\\
  21442 & 4.6290e-01 &  0.540 & 4.3878e-14\\
  84994 & 3.1816e-01 &  0.543 & 9.8787e-14\\
 338434 & 2.1844e-01 &  0.545 & 2.2136e-13\\
1350658 & 1.4988e-01 &  0.546 & 4.4909e-13\\
\end{tabular}
\caption{\label{tab:ex1brlr_1m}Errors for the modified Bernardi--Raugel finite element method for $\bm f = \nabla (\sin(xy\pi))$ and $\nu = 1$.}
\end{table}

\begin{table}
\centering
\begin{tabular}{ccccc}
ndof & $\! \| \nabla(\bv - \bv_h) \|_{L^2(\Omega)} \!$ & order & $ \!\| \nabla(\bv_h - \bS_h \bv) \|_{L^2(\Omega)} \!$ & order\\
\hline
    379 & 4.5794e+00 &  -     & 5.0351e+00 & -     \\
   1414 & 2.8168e+00 &  0.704 & 3.0576e+00 & 0.722 \\
   5458 & 1.5663e+00 &  0.850 & 1.6366e+00 & 0.905 \\
  21442 & 8.6350e-01 &  0.862 & 8.4114e-01 & 0.964 \\
  84994 & 4.8756e-01 &  0.828 & 4.2567e-01 & 0.986 \\
 338434 & 2.8682e-01 &  0.768 & 2.1402e-01 & 0.996 \\
1350658 & 1.7650e-01 &  0.703 & 1.0729e-01 & 1.000 \\
\end{tabular}
\caption{\label{tab:ex1brlr_2c}Errors for the classical Bernardi--Raugel finite element method for $\bm f = \nabla (\sin(xy\pi))$ and $\nu = 10^{-2}$.}
\end{table}

\begin{table}
\centering
\begin{tabular}{cccc}
ndof & $ \!\| \nabla(\bv - \bv_h) \|_{L^2(\Omega)} \!$ & order & $ \!\| \nabla(\bv_h - \bS_h \bv) \|_{L^2(\Omega)} \!$\\
\hline
    379 & 1.4142e+00 &  -     & 8.5800e-09\\
   1414 & 9.7261e-01 &  0.542 & 1.2516e-11\\
   5458 & 6.7218e-01 &  0.535 & 6.5365e-13\\
  21442 & 4.6290e-01 &  0.540 & 1.3425e-12\\
  84994 & 3.1816e-01 &  0.543 & 2.7291e-12\\
 338434 & 2.1844e-01 &  0.545 & 5.5018e-12\\
1350658 & 1.4988e-01 &  0.546 & 1.1034e-11\\
\end{tabular}
\caption{\label{tab:ex1brlr_2m}Errors for the modified Bernardi--Raugel finite element method for $\bm f = \nabla (\sin(xy\pi))$ and $\nu = 10^{-2}$.}
\end{table}

\begin{table}
\centering
\begin{tabular}{ccccc}
ndof & $\! \| \nabla(\bv - \bv_h) \|_{L^2(\Omega)} \!$ & order & $ \!\| \nabla(\bv_h - \bS_h \bv) \|_{L^2(\Omega)} \!$ & order\\
\hline
    379 & 4.3469e+02 &  -     & 5.0351e+02 & -     \\
   1414 & 2.6424e+02 &  0.721 & 3.0576e+02 & 0.722 \\
   5458 & 1.4134e+02 &  0.906 & 1.6366e+02 & 0.905 \\
  21442 & 7.2822e+01 &  0.960 & 8.4114e+01 & 0.964 \\
  84994 & 3.6908e+01 &  0.984 & 4.2567e+01 & 0.986 \\
 338434 & 1.8571e+01 &  0.995 & 2.1402e+01 & 0.996 \\
1350658 & 9.3137e+00 &  0.999 & 1.0729e+01 & 1.000 \\
\end{tabular}
\caption{\label{tab:ex1brlr_3c}Errors for the classical Bernardi--Raugel finite element method for $\bm f = \nabla (\sin(xy\pi))$ and $\nu = 10^{-4}$.}
\end{table}  

\begin{table}
\centering
\begin{tabular}{cccc}
ndof & $ \!\| \nabla(\bv - \bv_h) \|_{L^2(\Omega)} \!$ & order & $ \!\| \nabla(\bv_h - \bS_h \bv) \|_{L^2(\Omega)} \!$\\
\hline
    379 & 1.4142e+00 &  -     & 7.9830e-07\\
   1414 & 9.7261e-01 &  0.542 & 1.1388e-09\\
   5458 & 6.7218e-01 &  0.535 & 6.5440e-11\\
  21442 & 4.6290e-01 &  0.540 & 1.3407e-10\\
  84994 & 3.1816e-01 &  0.543 & 2.7257e-10\\
 338434 & 2.1844e-01 &  0.545 & 5.4872e-10\\
1350658 & 1.4988e-01 &  0.546 & 1.0965e-09\\
\end{tabular}
\caption{\label{tab:ex1brlr_3m}Errors for the modified Bernardi--Raugel finite element method for $\bm f = \nabla (\sin(xy\pi))$ and $\nu = 10^{-4}$.}
\end{table}

This sections gives a short numerical example to illustrate the theory.
We consider an L-shaped domain $\Omega:= (-1,1)^2 \setminus ((0,1) \times (-1,0))$ and the manufactured solution
\begin{align*}
 \bv(r,\varphi)
& :=r^\gamma
\begin{pmatrix}
(\gamma+1)\sin(\varphi)\psi(\varphi) + \cos(\varphi)\psi^\prime(\varphi)
\\
-(\gamma+1)\cos(\varphi)\psi(\varphi) + \sin(\varphi)\psi^\prime(\varphi)
\end{pmatrix}^T,\\
 p_0(r,\varphi) &:= \nu r^{(\gamma-1)}((1+\gamma)^2 \psi^\prime(\varphi)+\psi^{\prime\prime\prime}(\varphi))/(1-\gamma)
\end{align*}
where
\begin{align*}
\psi(\varphi) &:=
\frac{1}{\gamma+1} \, \sin((\gamma+1)\varphi)\cos(\gamma\omega) - \cos((\gamma+1)\varphi)\\
&\qquad- \frac{1}{\gamma-1} \, \sin((\gamma-1)\varphi)\cos(\gamma\omega) + \cos((\gamma-1)\varphi)
\end{align*}
and $\gamma = 856399/1572864 \approx 0.54$, $\omega = 3\pi/2$
taken from \cite{Verfuerth1989}.
Note, that this yields $-\nu\Delta \bv + \nabla p_0 = 0$.
To have a nonzero right-hand side we add $p_+ := \sin(xy\pi)$ to the pressure, i.e.\ $p := p_0 + p_+$ and $\bm f := \nabla(p_+)$.
Note that the exact solutions satisfy $\bv\in \bH^{1+s}(\Omega)$ and $p\in H^s(\Omega)$ for any $s<\gamma$.
Moreover, we set the viscosity parameter to either $\nu =1$, $\nu=10^{-2}$ or $\nu=10^{-4}$.

Tables~\ref{tab:ex1brlr_1c}-\ref{tab:ex1brlr_3m} compare the $\bH^1$ errors of the classical
Bernardi--Raugel finite element method and its pressure-robust sibling on a series of unstructured uniformly red-refined meshes for $\nu = 1$ (Tables~\ref{tab:ex1brlr_1c} and \ref{tab:ex1brlr_1m}), $\nu=10^{-2}$ (Tables~\ref{tab:ex1brlr_2c} and \ref{tab:ex1brlr_2m})
and $\nu=10^{-4}$ (Tables~\ref{tab:ex1brlr_3c} and \ref{tab:ex1brlr_3m}).
For the classical method the distance between the discrete Stokes projector and the discrete solution is non-zero
and really scales with $\nu^{-1}$, but asymptotically converges with $h$
instead of $h^s$. At first glance this seems better than expected
in Theorem~\ref{thm:ConsistencyErrorEstimate_classical},
but the first term vanishes due to $\mathbb{P}(-\Delta \bv) = \nu^{-1} \mathbb{P}(\nab p_0) = 0$ in this example. This also pre-asymptotically leads to a slightly higher convergence order of the full error than in case of
$\nu=1$ at least for $\nu=10^{-2}$
and $\nu=10^{-4}$ where the $\mathcal{O}(h)$
error dominates at first.
The numbers of the modified pressure-robust variant convey that the discrete solution of the modified method and the discrete Stokes projector are identical
as predicted by Lemma~\ref{thm:ConsistencyErrorEstimate} (again due to $\mathbb{P}(-\Delta \bv) = \nu^{-1} \mathbb{P}(\nab p_0) = 0$). The numerical results
confirm that for pressure-robust methods, the
discrete velocity is independent of $\nu$.
However, this $\nu$-independence only holds up to
a quadrature error in the right-hand side,  which scales with
$\nu^{-1}$, and up to round-off errors.

\bibliographystyle{abbrv}
\bibliography{lit}

\begin{thebibliography}{10}

\bibitem{MR3824769}
N.~Ahmed, A.~Linke, and C.~Merdon.
\newblock Towards pressure-robust mixed methods for the incompressible
  {N}avier-{S}tokes equations.
\newblock {\em Comput. Methods Appl. Math.}, 18(3):353--372, 2018.

\bibitem{BernardiRaugel}
C.~Bernardi and G.~Raugel.
\newblock Analysis of some finite elements for the {S}tokes problem.
\newblock {\em Math. Comp.}, 44(169):71--79, 1985.

\bibitem{MR3097958}
D.~Boffi, F.~Brezzi, and M.~Fortin.
\newblock {\em Mixed finite element methods and applications}, volume~44 of
  {\em Springer Series in Computational Mathematics}.
\newblock Springer, Heidelberg, 2013.

\bibitem{BLMS15}
C.~Brennecke, A.~Linke, C.~Merdon, and J.~Sch{\"o}berl.
\newblock Optimal and pressure-independent {$L^2$} velocity error estimates for
  a modified {C}rouzeix-{R}aviart {S}tokes element with {BDM} reconstructions.
\newblock {\em J. Comput. Math.}, 33(2):191--208, 2015.

\bibitem{crforty:brenner}
S.~C. Brenner.
\newblock Forty years of the {C}rouzeix--{R}aviart element.
\newblock {\em Numerical Methods for Partial Differential Equations},
  31(2):367--396, 2015.

\bibitem{crouzeix:raviart:1973}
M.~Crouzeix and P.-A. Raviart.
\newblock Conforming and nonconforming finite element methods for solving the
  stationary {S}tokes equations. {I}.
\newblock {\em Rev. Fran\c caise Automat. Informat. Recherche Op\'erationnelle
  S\'er. Rouge}, 7(R-3):33--75, 1973.

\bibitem{2018arXiv180810711G}
N.~R. {Gauger}, A.~{Linke}, and P.~W. {Schroeder}.
\newblock {On high-order pressure-robust space discretisations, their
  advantages for incompressible high Reynolds number generalised Beltrami flows
  and beyond}.
\newblock arXiv:1808.10711, Aug 2018.

\bibitem{GR86}
V.~Girault and P.-A. Raviart.
\newblock {\em Finite element methods for {N}avier--{S}tokes equations},
  volume~5 of {\em Springer Series in Computational Mathematics}.
\newblock Springer-Verlag, Berlin, 1986.
\newblock Theory and algorithms.

\bibitem{MR3120580}
J.~Guzm\'{a}n and M.~Neilan.
\newblock Conforming and divergence-free {S}tokes elements on general
  triangular meshes.
\newblock {\em Math. Comp.}, 83(285):15--36, 2014.

\bibitem{jlmnr:sirev}
V.~John, A.~Linke, C.~Merdon, M.~Neilan, and L.~G. Rebholz.
\newblock On the divergence constraint in mixed finite element methods for
  incompressible flows.
\newblock {\em SIAM Review}, 59(3):492--544, 2017.

\bibitem{LLMS2017}
P.~L. Lederer, A.~Linke, C.~Merdon, and J.~Sch\"{o}berl.
\newblock Divergence-free reconstruction operators for pressure-robust {S}tokes
  discretizations with continuous pressure finite elements.
\newblock {\em SIAM J. Numer. Anal.}, 55(3):1291--1314, 2017.

\bibitem{Lin14}
A.~Linke.
\newblock On the role of the {H}elmholtz decomposition in mixed methods for
  incompressible flows and a new variational crime.
\newblock {\em Comput. Methods Appl. Mech. Engrg.}, 268:782--800, 2014.

\bibitem{LMT15}
A.~Linke, G.~Matthies, and L.~Tobiska.
\newblock Robust arbitrary order mixed finite element methods for the
  incompressible {S}tokes equations with pressure independent velocity errors.
\newblock {\em ESAIM Math. Model. Numer. Anal.}, 50(1):289--309, 2016.

\bibitem{LM2016}
A.~Linke and C.~Merdon.
\newblock Pressure-robustness and discrete {H}elmholtz projectors in mixed
  finite element methods for the incompressible {N}avier--{S}tokes equations.
\newblock {\em Comput. Methods Appl. Mech. Engrg.}, 311:304--326, 2016.

\bibitem{MR3787384}
A.~Linke, C.~Merdon, M.~Neilan, and F.~Neumann.
\newblock Quasi-optimality of a pressure-robust nonconforming finite element
  method for the {S}tokes-problem.
\newblock {\em Math. Comp.}, 87(312):1543--1566, 2018.

\bibitem{scott:vogelius:conforming}
L.~R. Scott and M.~Vogelius.
\newblock Conforming finite element methods for incompressible and nearly
  incompressible continua.
\newblock In {\em Large-scale computations in fluid mechanics, {P}art 2 ({L}a
  {J}olla, {C}alif., 1983)}, volume~22 of {\em Lectures in Appl. Math.}, pages
  221--244. Amer. Math. Soc., Providence, RI, 1985.

\bibitem{Verfuerth1989}
R.~Verf\"{u}rth.
\newblock A posteriori error estimators for the {S}tokes equations.
\newblock {\em Numer. Math.}, 55(3):309--325, 1989.

\end{thebibliography}

\end{document}